\documentclass[a4paper,11pt]{article}

\usepackage{amsmath}
\usepackage{amssymb}
\usepackage{amsthm}
\usepackage{amsfonts}
\usepackage{mathtools}
\usepackage{xcolor}
\usepackage{url}
\numberwithin{equation}{section}

\theoremstyle{plain}
\newtheorem{theorem}{Theorem}[section]
\newtheorem{proposition}{Proposition}[section]
\newtheorem{lemma}{Lemma}[section]
\newtheorem{corollary}{Corollary}[section]

\theoremstyle{definition}
\newtheorem{definition}{Definition}[section]

\theoremstyle{remark}

\begin{document}
\title{\textbf{
Edge-Defect Spectral Methods for Higher-Order Rankings of Spanning Tree Counts}
\vspace{10mm}}

\author{
Shunya Tamura \\
Okegawa City Okegawa West Junior High School \\
3680-1 Kawataya, Okegawa, Saitama 363-0027, Japan \\
Corresponding author: 
shunya.tamura059@gmail.com \\
\\ \\
Jos\'e Luis Palacios \\
Department of Electrical and Computer Engineering, University of New Mexico \\
Albuquerque, NM 87131, USA \\
jpalacios@unm.edu
}

\date{}

\thispagestyle{empty}
\maketitle

\begin{abstract}
In this paper, we study the higher-order ranking, by spanning-tree count,
of graphs obtained from a complete graph by deleting a fixed number of edges.
Using the edge-defect matrix determined by the deleted edges and the
interaction number measuring the local overlap among them,
we derive a stability inequality that quantitatively estimates the decrease
in the number of spanning trees from the matching-deletion case.
Combining this stability estimate with a classification
of deletion graphs having small interaction number,
we determine, up to isomorphism, the nine deletion graphs with the
largest spanning-tree counts for $p\geq6$ and $n\geq2p$.
We also clarify the relation between the interaction number,
the local structure of the deletion graph, and the decrease in the number
of spanning trees through a logarithmic expansion of the normalized
spanning-tree count.

\end{abstract}

\noindent
{\bf Keywords:} spanning tree, edge-defect matrix, 
edge deletion, higher-order ranking, interaction number, Laplacian spectrum.

\noindent
{\bf 2020 Mathematics Subject Classification:}
Primary 05C50; Secondary 05C30, 15A18.

\section{Introduction}
\label{sec:introduction}

The number of spanning trees is one of the fundamental invariants
in graph theory.
If $\tau(G)$ denotes the number of spanning trees of a connected graph
$G$ on $n$ vertices, then, by Kirchhoff's matrix-tree theorem,
$\tau(G)$ can be expressed either as a cofactor of the Laplacian matrix
or in terms of the product of the nonzero Laplacian eigenvalues
\cite{Kirchhoff1847,Biggs1974}.
In particular, for the complete graph $K_n$,
Cayley's formula
\[
\tau(K_n)=n^{n-2}
\]
holds
\cite{Cayley1889}.

In addition to the problem of computing the number of spanning trees,
extremal problems concerning graphs that maximize the number of spanning trees
among graphs with prescribed numbers of vertices and edges
have long been studied.
Kelmans \cite{Kelmans1996},
Petingi, Boesch and Suffel \cite{PetingiBoeschSuffel1998},
and Petingi and Rodriguez \cite{PetingiRodriguez2002},
among others, investigated the structure of graphs maximizing
the number of spanning trees.
Moreover, Gilbert and Myrvold \cite{GilbertMyrvold1997} studied
comparison and maximization problems for the number of spanning trees
of almost complete graphs obtained by deleting a relatively small number
of edges from a complete graph.

For graphs obtained by deleting edges from a complete graph,
beginning with the classical work of Weinberg \cite{Weinberg1958},
formulas for the number of spanning trees associated with particular
deletion patterns, as well as more general determinantal representations,
have been studied
\cite{NikolopoulosPapadopoulos2006,GuoYan2023}.

Throughout the paper, let $n\geq3$, and let
\[
F=\{e_1,\ldots,e_p\}\subseteq E(K_n)
\]
be the set of deleted edges.
We also assume that
\[
G=K_n-F
\]
is connected.
Let $V(F)$ denote the set of all endpoints of the deleted edges, and define
\[
H_F=(V(F),F)
\]
to be the deletion graph associated with $F$.

Tamura \cite{Tamura2026EdgeDefect} introduced the edge-defect matrix
\[
Q=B^{\mathsf T}B,
\]
where the columns of $B$ are incidence vectors of arbitrarily oriented
deleted edges.

This matrix makes it possible to describe the effective resistance,
the Kirchhoff index, and the number of spanning trees of graphs obtained
by deleting edges from a complete graph within a unified spectral framework.
In particular, the following representation for the number of spanning trees
was obtained:
\[
\tau(K_n-F)
=
n^{\,n-p-2}\det(nI_p-Q)
\]
\cite{Tamura2026EdgeDefect}.

In Part I
\cite{TamuraPalacios2026},
we studied the higher-order ranking of the Kirchhoff index
for a fixed number $p$ of deleted edges.
As a quantity measuring the local overlap among the deleted edges,
we used the interaction number
\[
a(F)
=
\sum_{v\in V(H_F)}
\binom{d_{H_F}(v)}{2}
\]
and classified deletion graphs having small interaction number.
We further determined the first nine deletion graphs with the smallest
Kirchhoff indices under the assumptions
\[
p\geq6,
\qquad
n\geq\max\{13,2p\}.
\]

In the present paper, we apply this edge-defect spectral framework
to the higher-order ranking of the number of spanning trees.
More precisely, for fixed $p$, we aim to determine systematically,
among all graphs obtained from $K_n$ by deleting exactly $p$ edges,
not only the graph with the largest number of spanning trees,
but also those attaining the second, third, and subsequent largest values.

For this problem, we take the matching deletion
\[
H_F\cong pK_2
\]
as the reference configuration.
In this case,
\[
Q=2I_p,
\]
and hence
\[
A=Q-2I_p
\]
may be regarded as the spectral defect from the matching case.
Furthermore, defining
\[
T_{n,p}
=
n^{\,n-p-2}(n-2)^p
\]
as the reference value corresponding to the matching case and setting
\[
\mathcal{R}_{n,p}(F)
=
\frac{\tau(K_n-F)}{T_{n,p}},
\]
we obtain
\[
\mathcal{R}_{n,p}(F)
=
\det\left(
I_p-\frac{A}{n-2}
\right).
\]

The interaction number and the matrix $A$ satisfy
\[
\operatorname{tr}A=0,
\qquad
\operatorname{tr}A^2=2a(F)
\]
\cite{Tamura2026EdgeDefect,TamuraPalacios2026}.
By combining these relations with a log-determinant estimate,
we derive the following stability inequality for the normalized
number of spanning trees:
\[
\mathcal{R}_{n,p}(F)
\leq
\left(
\frac{n}{n-2}
\right)^{a(F)/2}
\exp\left(
-\frac{a(F)}{n-2}
\right).
\]
Thus, larger interaction number forces a stronger quantitative
decrease from the matching-deletion case.

Combining this stability estimate with the classification
of deletion graphs having small interaction number obtained in Part I
\cite{TamuraPalacios2026},
we obtain the main result of the present paper.
Namely, under the assumptions
$p\geq6$ and
\[
n\geq2p
\]
we show that the nine largest spanning tree counts are attained,
uniquely up to isomorphism and in the following order,
when the deletion graph is isomorphic to
\[
\begin{array}{
r@{\;}l
@{\qquad\qquad}
r@{\;}l
@{\qquad\qquad}
r@{\;}l
}
\text{(1)} & pK_2
&
\text{(2)} & P_3\cup(p-2)K_2
&
\text{(3)} & 2P_3\cup(p-4)K_2
\\[10pt]
\text{(4)} & P_4\cup(p-3)K_2
&
\text{(5)} & K_3\cup(p-3)K_2
&
\text{(6)} & 3P_3\cup(p-6)K_2
\\[10pt]
\text{(7)} & P_4\cup P_3\cup(p-5)K_2
&
\text{(8)} & P_5\cup(p-4)K_2
&
\text{(9)} & K_{1,3}\cup(p-3)K_2
\end{array}
\]
All other deletion graphs are excluded from the first nine positions
by the stability inequality.
Thus, only finitely many low-interaction candidates require direct comparison.

Throughout the range considered in Part I,
\[
p\geq6,
\qquad
n\geq\max\{13,2p\},
\]
the ordering of the nine largest spanning-tree counts obtained here
coincides with that of the nine smallest Kirchhoff indices in
\cite{TamuraPalacios2026}.
This coincidence is not immediate, since the Kirchhoff index is expressed
as a sum of rational functions of the edge-defect eigenvalues,
whereas the number of spanning trees is described by their product.
Our result shows that, at least in the above range,
the initial rankings of these two different graph invariants
are governed by the same local structures of the deletion graph.

Furthermore, by considering the logarithm of the normalized number
of spanning trees, we expand
\[
-\log\mathcal{R}_{n,p}(F)
\]
in terms of higher-order traces of $A$.
The quadratic term is determined by the interaction number,
while the cubic term involves the difference between the number of triples
of deleted edges incident with a common vertex and the number of triangles
contained in the deletion graph.
Consequently, for a fixed deletion graph and in the asymptotic regime
$n\to\infty$, the interaction number appears as the first local structural
quantity governing the decrease in the number of spanning trees,
whereas the higher-order traces detect finer distinctions among deletion graphs
having the same interaction number.

The paper is organized as follows.
In Section \ref{sec:preliminaries}, we collect the necessary known results
and notation and introduce the normalized number of spanning trees.
In Section \ref{sec:stability}, we derive the stability inequality,
and in Section \ref{sec:candidates}, we compare the normalized spanning-tree
counts of the nine candidate configurations.
In Section \ref{sec:main}, we exclude all cases with $a(F)\geq4$
simultaneously and prove the main theorem.
Finally, in Section \ref{sec:log-expansion},
we use the log-determinant expansion to investigate the relation
between the number of spanning trees and the local structure
of the deletion graph.

\section{Preliminaries}
\label{sec:preliminaries}

In this section, we collect the notation and known results needed for the proof
of the main theorem.
Unless otherwise stated, all graphs considered in this paper are finite,
simple, and undirected.

For a graph $G$, let $V(G)$ and $E(G)$ denote its vertex set and edge set,
respectively, and let $\tau(G)$ denote the number of spanning trees of $G$.
Let $L(G)$ denote the Laplacian matrix of $G$.
When $G$ is a connected graph on $n$ vertices,
we write its Laplacian eigenvalues as
\[
0=\lambda_1<\lambda_2\leq\cdots\leq\lambda_n.
\]

The number of spanning trees and the Laplacian spectrum are related by
the following standard consequence of Kirchhoff's Matrix--Tree Theorem.

\begin{theorem}[\cite{Kirchhoff1847,Biggs1974}, Matrix--Tree Theorem]
\label{thm:matrix-tree}
For a connected graph $G$ on $n$ vertices,
\[
\tau(G)
=
\frac{1}{n}
\prod_{i=2}^{n}\lambda_i.
\]
\end{theorem}

In particular, since all nonzero Laplacian eigenvalues of the complete graph
$K_n$ are equal to $n$, we obtain
\[
\tau(K_n)=n^{n-2}.
\]
This is Cayley's formula \cite{Cayley1889}.

Throughout the remainder of the paper, let $n\geq3$, let
\[
F=\{e_1,\ldots,e_p\}\subseteq E(K_n)
\]
be the set of deleted edges, and assume that
\[
G=K_n-F
\]
is connected.
We also set
\[
V(F)
=
\bigcup_{e\in F}V(e),
\qquad
H_F=(V(F),F),
\]
and call $H_F$ the deletion graph associated with $F$.

Assign an arbitrary orientation to each deleted edge
\[
e_\alpha=\{u_\alpha,v_\alpha\},
\]
and let the corresponding incidence vector be
\[
b_\alpha=e_{u_\alpha}-e_{v_\alpha}.
\]
Arranging these vectors as columns, define
\[
B=
\begin{pmatrix}
b_1&\cdots&b_p
\end{pmatrix}
\in\mathbb{R}^{n\times p}.
\]

The following matrix was introduced by Tamura \cite{Tamura2026EdgeDefect}.
\begin{definition}[\cite{Tamura2026EdgeDefect}, Definition 2.1]
\label{def:edge-defect}
For a set $F$ of deleted edges,
\[
Q=Q_F=B^{\mathsf T}B
\]
is called the edge-defect matrix of $F$.
\end{definition}

We recall the basic properties of the edge-defect matrix.
\begin{lemma}[\cite{Tamura2026EdgeDefect}, Lemma 2.3]
\label{lem:basic-Q}
The edge-defect matrix $Q$ is a real symmetric positive semidefinite matrix
satisfying
\[
Q_{\alpha\alpha}=2,
\qquad
\operatorname{tr}Q=2p.
\]
Moreover, for $\alpha\neq\beta$,
\[
Q_{\alpha\beta}
=
\begin{cases}
0,
&
e_\alpha\cap e_\beta=\varnothing,
\\
\pm1,
&
|e_\alpha\cap e_\beta|=1.
\end{cases}
\]
Furthermore, the eigenvalues of $Q$ are independent of the choice
of orientations of the deleted edges.
\end{lemma}

The connectivity of $K_n-F$ also yields a bound on the edge-defect
eigenvalues.

\begin{lemma}[\cite{Tamura2026EdgeDefect}, Lemma 2.5]
\label{lem:defect-bound}
Suppose that $K_n-F$ is connected.
Then
\[
nI_p-Q
\]
is positive definite.
Consequently, every eigenvalue $\theta$ of $Q$ satisfies
\[
0\leq\theta<n.
\]
\end{lemma}

The spectrum of the edge-defect matrix can be described in terms of
the Laplacian spectrum of the deletion graph.

\begin{lemma}[\cite{Tamura2026EdgeDefect}, Lemma 2.6]
\label{lem:defect-spectrum}
Suppose that $H_F$ has $q$ vertices, $p$ edges, and $c$ connected components,
and let its positive Laplacian eigenvalues be
\[
\lambda_1,\ldots,\lambda_{q-c}.
\]
Then the eigenvalues of $Q$ are
\[
\lambda_1,\ldots,\lambda_{q-c}
\]
together with
\[
\underbrace{0,\ldots,0}_{p-q+c\text{ copies}}.
\]
\end{lemma}

Moreover, a matching deletion is characterized by the edge-defect matrix
as follows.

\begin{proposition}[\cite{Tamura2026EdgeDefect}, Proposition 2.1]
\label{prop:matching-characterization}
The following three conditions are equivalent:
\[
F\text{ is a matching},
\]
\[
H_F\cong pK_2,
\]
\[
Q=2I_p.
\]
\end{proposition}
The starting point for our study of the number of spanning trees is
the following spectral formula obtained by Tamura
\cite{Tamura2026EdgeDefect}.
\begin{theorem}[\cite{Tamura2026EdgeDefect}, Theorem 4.2]
\label{thm:tree-defect-formula}
Suppose that $K_n-F$ is connected, and let the eigenvalues of $Q$ be
\[
\theta_1,\ldots,\theta_p.
\]
Then
\[
\tau(K_n-F)
=
n^{\,n-2}
\prod_{\alpha=1}^{p}
\left(
1-\frac{\theta_\alpha}{n}
\right)
=
n^{\,n-p-2}\det(nI_p-Q).
\]
\end{theorem}

We next introduce a quantity measuring the local overlap among the deleted
edges.
We use the quantity introduced in Tamura \cite{Tamura2026EdgeDefect}
and called the interaction number in Part I
\cite{TamuraPalacios2026}:
\[
a(F)
=
\sum_{v\in V(H_F)}
\binom{d_{H_F}(v)}{2}.
\]
This is the total number of unordered pairs of deleted edges sharing
a common endpoint.

Furthermore, set
\[
A=Q-2I_p.
\]
The interaction number and the quadratic trace of $A$ are related
as follows.

\begin{lemma}[\cite{Tamura2026EdgeDefect}, Lemma 5.1]
\label{lem:interaction-trace}
\[
\operatorname{tr}A^2
=
2a(F).
\]
\end{lemma}
In Part I, deletion graphs with small interaction number were
completely classified.
We recall the classification needed here.

\begin{proposition}[\cite{TamuraPalacios2026}, Proposition 3.1]
\label{prop:small-interaction}
Suppose that $H_F$ has $p\geq6$ edges.
Then the following hold.

\begin{enumerate}

\item
$a(F)=0$ if and only if $H_F\cong pK_2$.

\item
$a(F)=1$ if and only if $H_F\cong P_3\cup(p-2)K_2$.

\item
$a(F)=2$ if and only if $H_F\cong2P_3\cup(p-4)K_2$ or $H_F\cong P_4\cup(p-3)K_2$.

\item
$a(F)=3$ if and only if $H_F$ is isomorphic to one of
\[
K_3\cup(p-3)K_2, \qquad
3P_3\cup(p-6)K_2, \qquad
P_4\cup P_3\cup(p-5)K_2,
\]
\[
P_5\cup(p-4)K_2, \qquad
K_{1,3}\cup(p-3)K_2.
\]

\end{enumerate}
\end{proposition}
We now introduce the normalization used throughout the paper.

By Proposition \ref{prop:matching-characterization} and
Theorem \ref{thm:tree-defect-formula},
whenever a matching deletion is realizable,
if $H_F\cong pK_2$, then
\[
\tau(K_n-F)
=
n^{\,n-p-2}(n-2)^p.
\]

Accordingly, we define
\[
T_{n,p}
=
n^{\,n-p-2}(n-2)^p
\]
as the reference value corresponding to the matching case, and define
the normalized spanning-tree count by
\[
\mathcal{R}_{n,p}(F)
=
\frac{\tau(K_n-F)}{T_{n,p}}.
\]

This normalized quantity admits a simple determinant representation
in terms of $A=Q-2I_p$.
\begin{lemma}
\label{lem:normalized-tree}
Suppose that $K_n-F$ is connected.
Then
\[
\mathcal{R}_{n,p}(F)
=
\det\left(
I_p-\frac{A}{n-2}
\right).
\]
Consequently, if the eigenvalues of $A$ are $\mu_1,\ldots,\mu_p$, then
\[
\mathcal{R}_{n,p}(F)
=
\prod_{\alpha=1}^{p}
\left(
1-\frac{\mu_\alpha}{n-2}
\right).
\]
\end{lemma}

\begin{proof}
By Theorem \ref{thm:tree-defect-formula} and
\[
A=Q-2I_p,
\]
we have
\[
nI_p-Q
=
(n-2)I_p-A.
\]
Therefore,
\[
\begin{aligned}
\mathcal{R}_{n,p}(F)
&=
\frac{
n^{\,n-p-2}\det(nI_p-Q)
}{
n^{\,n-p-2}(n-2)^p
}
\\
&=
\frac{
\det\bigl((n-2)I_p-A\bigr)
}{
(n-2)^p
}
\\
&=
\det\left(
I_p-\frac{A}{n-2}
\right).
\end{aligned}
\]
If the eigenvalues of $A$ are
$\mu_1,\ldots,\mu_p$, then expressing the determinant as the product
of the eigenvalues yields
\[
\mathcal{R}_{n,p}(F)
=
\prod_{\alpha=1}^{p}
\left(
1-\frac{\mu_\alpha}{n-2}
\right).
\]
\end{proof}
Hereafter, if $H$ is a deletion graph realizable as a subgraph of $K_n$
and there exists a set of deleted edges $F$ such that $H_F\cong H$, we write
\[
\mathcal{R}_n(H)
=
\mathcal{R}_{n,|E(H)|}(F).
\]
We also write
\[
\tau(K_n-H)
=
\tau(K_n-F).
\]
These quantities depend only on the isomorphism class of the deletion graph
$H$.

We next record several relations concerning the eigenvalues of $A$ that
will be used later.
\begin{lemma}
\label{lem:defect-moments}
Let $\mu_1,\ldots,\mu_p$ be the eigenvalues of $A=Q-2I_p$.
Then
\[
-2\leq\mu_\alpha<n-2
\qquad
(1\leq\alpha\leq p),
\qquad
\sum_{\alpha=1}^{p}\mu_\alpha=0,
\qquad
\sum_{\alpha=1}^{p}\mu_\alpha^2
=2a(F)
\]
hold.
\end{lemma}

\begin{proof}
Since $\mu_\alpha=\theta_\alpha-2$ and
$0\leq\theta_\alpha<n$ by Lemma \ref{lem:defect-bound},
we have
\[
-2\leq\mu_\alpha<n-2.
\]

Moreover, by Lemma \ref{lem:basic-Q},
\[
\operatorname{tr}Q=2p,
\]
and hence
\[
\operatorname{tr}A
=
\operatorname{tr}(Q-2I_p)
=
0.
\]
Therefore,
\[
\sum_{\alpha=1}^{p}\mu_\alpha=0.
\]

Furthermore, by Lemma \ref{lem:interaction-trace},
\[
\operatorname{tr}A^2=2a(F).
\]
Since $A$ is a real symmetric matrix,
\[
\operatorname{tr}A^2
=
\sum_{\alpha=1}^{p}\mu_\alpha^2.
\]
Thus,
\[
\sum_{\alpha=1}^{p}\mu_\alpha^2
=
2a(F).
\]
\end{proof}

\begin{lemma}
\label{lem:multiplicativity}
Suppose that $H_1$ and $H_2$ are vertex-disjoint deletion graphs.
Then
\[
\mathcal{R}_n(H_1\cup H_2)
=
\mathcal{R}_n(H_1)
\mathcal{R}_n(H_2).
\]
In particular, $\mathcal{R}_n(K_2)=1$.
\end{lemma}

\begin{proof}
Since $H_1$ and $H_2$ are vertex-disjoint,
by choosing a suitable ordering of the edges, the corresponding
edge-defect matrix is block diagonal.
Therefore,
\[
A(H_1\cup H_2)
=
A(H_1)\oplus A(H_2).
\]
By Lemma \ref{lem:normalized-tree},
\[
\begin{aligned}
\mathcal{R}_n(H_1\cup H_2)
&=
\det\left(
I-\frac{A(H_1\cup H_2)}{n-2}
\right)
\\
&=
\det\left(
I-\frac{A(H_1)}{n-2}
\right)
\det\left(
I-\frac{A(H_2)}{n-2}
\right)
\\
&=
\mathcal{R}_n(H_1)
\mathcal{R}_n(H_2).
\end{aligned}
\]

Moreover, since $Q(K_2)=(2)$,
\[
A(K_2)=(0).
\]
Therefore,
\[
\mathcal{R}_n(K_2)=1.
\]
\end{proof}

In the next section, we use these relations to obtain a quantitative estimate
for the decrease in the number of spanning trees from the matching-deletion
case.

\section{Stability of the normalized spanning tree count}
\label{sec:stability}

In this section, we use the interaction number $a(F)$ to obtain a
quantitative estimate for the decrease in the normalized spanning-tree count
from the matching case.

We continue to use the notation introduced in the preceding section.
We first establish an elementary inequality needed to estimate the product
representation in Lemma \ref{lem:normalized-tree}.

\begin{lemma}
\label{lem:log-bound}

Let $r>0$, and define
\[
c_r
=
\frac{r-\log(1+r)}{r^2}.
\]
Then, for every $-r\leq t<1$,
\[
\log(1-t)
\leq
-t-c_rt^2.
\]

\end{lemma}

\begin{proof}

First, suppose that $0<t<1$.
Then
\[
-t-\log(1-t)
=
\int_0^t\frac{s}{1-s}\,ds
\geq
\int_0^t s\,ds
=
\frac{t^2}{2}.
\]

On the other hand,
\[
r-\log(1+r)
=
\int_0^r\frac{s}{1+s}\,ds
<
\int_0^r s\,ds
=
\frac{r^2}{2},
\]
and hence
\[
c_r<\frac12.
\]
Therefore,
\[
-t-\log(1-t)
\geq
c_rt^2,
\]
or equivalently,
\[
\log(1-t)
\leq
-t-c_rt^2.
\]

Next, suppose that $-r\leq t<0$.
Writing $t=-s$, we have $0<s\leq r$. 

Define
\[
\phi(s)
=
\frac{s-\log(1+s)}{s^2}
\qquad
(s>0).
\]

A direct calculation gives
\[
\phi'(s)
=
\frac{
2(1+s)\log(1+s)-2s-s^2
}{
s^3(1+s)
}.
\]

Now define
\[
h(s)
=
2s+s^2-2(1+s)\log(1+s).
\]
Then $h(0)=0$, and
\[
h'(s)
=
2s-2\log(1+s)>0
\qquad
(s>0).
\]
Therefore,
\[
h(s)>0
\qquad
(s>0),
\]
and hence
\[
\phi'(s)<0.
\]

Thus, $\phi$ is decreasing, and hence, for $0<s\leq r$,
\[
\phi(s)
\geq
\phi(r)
=
\frac{r-\log(1+r)}{r^2}
=
c_r.
\]
Therefore,
\[
s-\log(1+s)
\geq
c_rs^2.
\]
Returning to $t=-s$, we obtain
\[
\log(1-t)
\leq
-t-c_rt^2.
\]

The case $t=0$ is immediate.

\end{proof}

The following theorem provides the basic stability estimate used in this paper.

\begin{theorem}
\label{thm:tree-stability}

Suppose that $K_n-F$ is connected.
Then
\[
\log\mathcal{R}_{n,p}(F)
\leq
-\frac{a(F)}{2}
\left\{
\frac{2}{n-2}
-
\log\left(\frac{n}{n-2}\right)
\right\}.
\]

Consequently,
\[
\mathcal{R}_{n,p}(F)
\leq
\left(
\frac{n}{n-2}
\right)^{a(F)/2}
\exp\left(
-\frac{a(F)}{n-2}
\right).
\]

\end{theorem}

\begin{proof}
Set $x=n-2$.
By Lemma \ref{lem:defect-moments},
\[
-2\leq\mu_\alpha<x.
\]
Therefore,
\[
-\frac{2}{x}
\leq
\frac{\mu_\alpha}{x}
<1.
\]

Applying Lemma \ref{lem:log-bound} with $r=2/x$ and $t=\mu_\alpha/x$, we obtain
\[
\log\left(
1-\frac{\mu_\alpha}{x}
\right)
\leq
-\frac{\mu_\alpha}{x}
-
c_{2/x}\frac{\mu_\alpha^2}{x^2}.
\]

Summing over
$\alpha=1,\ldots,p$,
we obtain
\[
\begin{aligned}
\log\mathcal{R}_{n,p}(F)
&=
\sum_{\alpha=1}^{p}
\log\left(
1-\frac{\mu_\alpha}{x}
\right)
\\
&\leq
-\frac1x
\sum_{\alpha=1}^{p}\mu_\alpha
-
\frac{c_{2/x}}{x^2}
\sum_{\alpha=1}^{p}\mu_\alpha^2.
\end{aligned}
\]

By Lemma \ref{lem:defect-moments},
\[
\sum_{\alpha=1}^{p}\mu_\alpha=0,
\qquad
\sum_{\alpha=1}^{p}\mu_\alpha^2=2a(F),
\]
and hence
\[
\log\mathcal{R}_{n,p}(F)
\leq
-\frac{2a(F)c_{2/x}}{x^2}.
\]

Now,
\[
c_{2/x}
=
\frac{
\frac2x-\log\left(1+\frac2x\right)
}{
\frac4{x^2}
}.
\]
Therefore,
\[
\log\mathcal{R}_{n,p}(F)
\leq
-\frac{a(F)}{2}
\left\{
\frac2x
-
\log\left(1+\frac2x\right)
\right\}.
\]

Since $x=n-2$,
\[
1+\frac2x
=
\frac{n}{n-2}.
\]
Thus,
\[
\log\mathcal{R}_{n,p}(F)
\leq
-\frac{a(F)}{2}
\left\{
\frac{2}{n-2}
-
\log\left(\frac{n}{n-2}\right)
\right\}.
\]

Exponentiating both sides yields
\[
\mathcal{R}_{n,p}(F)
\leq
\left(
\frac{n}{n-2}
\right)^{a(F)/2}
\exp\left(
-\frac{a(F)}{n-2}
\right).
\]

\end{proof}

For the proof of the main theorem, the following slightly weaker but simpler
estimate will be convenient.

\begin{corollary}
\label{cor:simple-stability}

Suppose that $K_n-F$ is connected.
Then
\[
\mathcal{R}_{n,p}(F)
\leq
\exp\left(
-\frac{a(F)}{n(n-2)}
\right).
\]

\end{corollary}

\begin{proof}

For every $r>0$,
\[
r-\log(1+r)
=
\int_0^r\frac{s}{1+s}\,ds.
\]
Since
\[
\frac{s}{1+s}
\geq
\frac{s}{1+r}
\]
for $0\leq s\leq r$, we have
\[
r-\log(1+r)
\geq
\frac1{1+r}\int_0^r s\,ds
=
\frac{r^2}{2(1+r)}.
\]

Taking $r=2/(n-2)$ gives
\[
\frac12
\left\{
\frac{2}{n-2}
-
\log\left(\frac{n}{n-2}\right)
\right\}
\geq
\frac{1}{n(n-2)}.
\]

Therefore, by Theorem \ref{thm:tree-stability},
\[
\log\mathcal{R}_{n,p}(F)
\leq
-\frac{a(F)}{n(n-2)}.
\]
Hence,
\[
\mathcal{R}_{n,p}(F)
\leq
\exp\left(
-\frac{a(F)}{n(n-2)}
\right).
\]

\end{proof}

In particular, $\mathcal{R}_{n,p}(F)\leq1$.
Moreover, equality holds if and only if $a(F)=0$, equivalently,
if and only if $H_F\cong pK_2$.

\section{Normalized spanning-tree counts and rankings of the candidate graphs}
\label{sec:candidates}

Set $x=n-2$.
By Lemma \ref{lem:multiplicativity},
the normalized spanning-tree count is multiplicative
with respect to disjoint unions, and
\[
\mathcal{R}_n(K_2)=1.
\]
Hence the isolated $K_2$ components do not affect the normalized count.
It therefore suffices to compute
\[
\mathcal{R}_n(P_3),\qquad
\mathcal{R}_n(P_4),\qquad
\mathcal{R}_n(K_3),\qquad
\mathcal{R}_n(P_5),\qquad
\mathcal{R}_n(K_{1,3}).
\]

\begin{proposition}
\label{prop:basic-ratios}
Let $n\geq5$, so that all the deletion graphs considered below
are realizable as subgraphs of $K_n$, and set $x=n-2$.
Then
\[
\mathcal{R}_n(P_3)
=1-\frac{1}{x^2},
\qquad
\mathcal{R}_n(P_4)
=1-\frac{2}{x^2},
\qquad
\mathcal{R}_n(K_3)
=1-\frac{3}{x^2}+\frac{2}{x^3},
\]
\[
\mathcal{R}_n(P_5)
=1-\frac{3}{x^2}+\frac{1}{x^4},
\qquad
\mathcal{R}_n(K_{1,3})
=1-\frac{3}{x^2}-\frac{2}{x^3}.
\]

\end{proposition}

\begin{proof}

By Lemma \ref{lem:defect-spectrum},
the nonzero eigenvalues of the edge-defect matrix coincide with
the nonzero Laplacian eigenvalues of the corresponding deletion graph.

First, the nonzero Laplacian eigenvalues of $P_3$ are $\{ 1,\ 3\}$.
Therefore, the eigenvalues of $A=Q-2I$ are $\{-1,\ 1\}$, and hence,
by Lemma \ref{lem:normalized-tree},
\[
\begin{aligned}
\mathcal{R}_n(P_3)
&=
\left(1+\frac1x\right)
\left(1-\frac1x\right)
\\
&=
1-\frac1{x^2}.
\end{aligned}
\]

Next, the nonzero Laplacian eigenvalues of $P_4$ are $\{2-\sqrt2,\ 2,\ 2+\sqrt2\}$.
Therefore, the eigenvalues of $A$ are $\{-\sqrt2,\ 0,\ \sqrt2\}$, and
\[
\begin{aligned}
\mathcal{R}_n(P_4)
&=
\left(1+\frac{\sqrt2}{x}\right)
\left(1-\frac{\sqrt2}{x}\right)
\\
&=
1-\frac2{x^2}.
\end{aligned}
\]

The Laplacian eigenvalues of $K_3$ are $\{0,\ 3,\ 3\}$.
Since $K_3$ has $3$ edges,
Lemma \ref{lem:defect-spectrum} implies that
the eigenvalues of $Q$ are $\{0,\ 3,\ 3\}$.
Therefore, the eigenvalues of $A$ are $\{-2,\ 1,\ 1\}$,
and
\[
\begin{aligned}
\mathcal{R}_n(K_3)
&=
\left(1+\frac2x\right)
\left(1-\frac1x\right)^2
\\
&=
1-\frac3{x^2}+\frac2{x^3}.
\end{aligned}
\]

Next, the nonzero Laplacian eigenvalues of $P_5$ are
\[
\frac{3-\sqrt5}{2},
\qquad
\frac{5-\sqrt5}{2},
\qquad
\frac{3+\sqrt5}{2},
\qquad
\frac{5+\sqrt5}{2}.
\]
Therefore, the eigenvalues of $A$ are
\[
-\frac{1+\sqrt5}{2},
\qquad
\frac{1-\sqrt5}{2},
\qquad
\frac{\sqrt5-1}{2},
\qquad
\frac{1+\sqrt5}{2}.
\]
Using these eigenvalues, we obtain
\[
\begin{aligned}
\mathcal{R}_n(P_5)
&=
\prod_{\alpha=1}^{4}
\left(
1-\frac{\mu_\alpha}{x}
\right)
\\
&=
\left(
1-\frac{3+\sqrt5}{2x^2}
\right)
\left(
1-\frac{3-\sqrt5}{2x^2}
\right)
\\
&=
1-\frac3{x^2}+\frac1{x^4}.
\end{aligned}
\]

Finally, the nonzero Laplacian eigenvalues of $K_{1,3}$ are $\{1,\ 1,\ 4\}$.
Therefore, the eigenvalues of $A$ are $\{-1,\ -1,\ 2\}$, and
\[
\begin{aligned}
\mathcal{R}_n(K_{1,3})
&=
\left(1+\frac1x\right)^2
\left(1-\frac2x\right)
\\
&=
1-\frac3{x^2}-\frac2{x^3}.
\end{aligned}
\]

This proves the proposition.

\end{proof}

Combining Proposition \ref{prop:small-interaction} with
Lemma \ref{lem:multiplicativity},
we obtain the normalized spanning-tree counts of all deletion graphs
satisfying $a(F)\leq3$.

\begin{corollary}
\label{cor:nine-ratios}

Let $p\geq6$ and $n\geq2p$, and set $x=n-2$.
Consider the following nine deletion graphs:
\[
\begin{array}{
r@{\;}c@{\;}l
@{\qquad\qquad}
r@{\;}c@{\;}l
@{\qquad\qquad}
r@{\;}c@{\;}l
}
H_1 & = & pK_2
&
H_2 & = & P_3\cup(p-2)K_2
&
H_3 & = & 2P_3\cup(p-4)K_2
\\[10pt]
H_4 & = & P_4\cup(p-3)K_2
&
H_5 & = & K_3\cup(p-3)K_2
&
H_6 & = & 3P_3\cup(p-6)K_2
\\[10pt]
H_7 & = & P_4\cup P_3\cup(p-5)K_2
&
H_8 & = & P_5\cup(p-4)K_2
&
H_9 & = & K_{1,3}\cup(p-3)K_2
\end{array}
\]
Then
\[
\begin{array}{
r@{\;}c@{\;}l
@{\quad\quad}
r@{\;}c@{\;}l
@{\quad\quad}
r@{\;}c@{\;}l
}
\mathcal{R}_n(H_1) & = & 1
&
\mathcal{R}_n(H_2) & = & 1-\dfrac1{x^2}
&
\mathcal{R}_n(H_3) & = & \left(1-\dfrac1{x^2}\right)^2
\\[12pt]
\mathcal{R}_n(H_4) & = & 1-\dfrac2{x^2}
&
\mathcal{R}_n(H_5) & = & 1-\dfrac3{x^2}+\dfrac2{x^3}
&
\mathcal{R}_n(H_6) & = & \left(1-\dfrac1{x^2}\right)^3
\\[12pt]
\mathcal{R}_n(H_7) & = &
\left(1-\dfrac2{x^2}\right)
\left(1-\dfrac1{x^2}\right)
&
\mathcal{R}_n(H_8) & = & 1-\dfrac3{x^2}+\dfrac1{x^4}
&
\mathcal{R}_n(H_9) & = & 1-\dfrac3{x^2}-\dfrac2{x^3}
\end{array}
\]
\end{corollary}

\begin{proof}

It suffices to apply Lemma \ref{lem:multiplicativity} and
\[
\mathcal{R}_n(K_2)=1
\]
to Proposition \ref{prop:basic-ratios}.

\end{proof}

Now set
\[
\rho_i=\mathcal{R}_n(H_i)
\qquad
(1\leq i\leq9).
\]

The following proposition completely determines the ranking of these nine candidates.

\begin{proposition}
\label{prop:candidate-order}

Let $p\geq6$ and $n\geq2p$.
Then
\[
\rho_1>\rho_2>\rho_3>\rho_4>\rho_5>
\rho_6>\rho_7>\rho_8>\rho_9.
\]

\end{proposition}

\begin{proof}

Set $x=n-2$.
Since $p\geq6$ and $n\geq2p$,
\[
n\geq12,
\qquad
x\geq10>2.
\]

Using the formulas in Corollary \ref{cor:nine-ratios},
a direct calculation of the successive differences gives
\[
\begin{array}{
r@{\;}c@{\;}l
@{\qquad\qquad}
r@{\;}c@{\;}l
}
\rho_1-\rho_2 & = & \dfrac1{x^2}>0
&
\rho_2-\rho_3 & = & \dfrac{x^2-1}{x^4}>0
\\[12pt]
\rho_3-\rho_4 & = & \dfrac1{x^4}>0
&
\rho_4-\rho_5 & = & \dfrac{x-2}{x^3}>0
\\[12pt]
\rho_5-\rho_6 & = & \dfrac{(x-1)^2(2x+1)}{x^6}>0
&
\rho_6-\rho_7 & = & \dfrac{x^2-1}{x^6}>0
\\[12pt]
\rho_7-\rho_8 & = & \dfrac1{x^4}>0
&
\rho_8-\rho_9 & = & \dfrac{2x+1}{x^4}>0
\end{array}
\]
Therefore,
\[
\rho_1>\rho_2>\rho_3>\rho_4>\rho_5>
\rho_6>\rho_7>\rho_8>\rho_9.
\]

\end{proof}

\section{Exclusion of large interaction numbers and the main theorem}
\label{sec:main}
The case $a(F)\leq3$ is completely determined by
Proposition \ref{prop:small-interaction} and
Proposition \ref{prop:candidate-order}.
It remains to exclude deletion graphs with $a(F)\geq4$.

\begin{proposition}
\label{prop:large-interaction}

Let $n\geq12$, and suppose that
$K_n-F$ is connected.
If $a(F)\geq4$, then
\[
\mathcal{R}_{n,p}(F)
<
\mathcal{R}_n(K_{1,3}).
\]

\end{proposition}

\begin{proof}

Set $x=n-2$.
Since $n\geq12$, we have $x\geq10$.

By Corollary \ref{cor:simple-stability},
\[
\mathcal{R}_{n,p}(F)
\leq
\exp\left(
-\frac{a(F)}{n(n-2)}
\right).
\]
Since $a(F)\geq4$,
\[
\mathcal{R}_{n,p}(F)
\leq
\exp\left(
-\frac{4}{x(x+2)}
\right).
\]

Now set
\[
y=\frac{4}{x(x+2)}>0.
\]

For every $y>0$,
\[
e^{-y}
<
1-y+\frac{y^2}{2},
\]
and hence
\[
\mathcal{R}_{n,p}(F)
<
1-\frac{4}{x(x+2)}
+
\frac{8}{x^2(x+2)^2}.
\]

On the other hand, by Proposition \ref{prop:basic-ratios},
\[
\mathcal{R}_n(K_{1,3})
=
1-\frac{3}{x^2}-\frac{2}{x^3}.
\]
Therefore,
\[
\begin{aligned}
\mathcal{R}_n(K_{1,3})
-
\left\{
1-\frac{4}{x(x+2)}
+
\frac{8}{x^2(x+2)^2}
\right\}
&=
\frac{x^3-6x^2-28x-8}
{x^3(x+2)^2}.
\end{aligned}
\]

Set
\[
f(x)=x^3-6x^2-28x-8.
\]

At $x=10$,
\[
f(10)=112>0.
\]
Moreover,
\[
f'(x)=3x^2-12x-28.
\]
For $x\geq10$, we have $f'(x)>0$, and hence
\[
f(x)>0
\qquad
(x\geq10).
\]

Therefore,
\[
1-\frac{4}{x(x+2)}
+
\frac{8}{x^2(x+2)^2}
<
1-\frac{3}{x^2}-\frac{2}{x^3}
=
\mathcal{R}_n(K_{1,3}).
\]

Consequently,
\[
\mathcal{R}_{n,p}(F)
<
\mathcal{R}_n(K_{1,3}),
\]
as required.

\end{proof}

Thus, all deletion graphs with interaction number at least $4$
can be excluded from the first nine positions simultaneously,
without classifying them individually.

Combining the preceding results yields the main theorem of this paper.

\begin{theorem}
\label{thm:main-ranking}

Let $p\geq6$ and $n\geq2p$, and let
$H_1,\ldots,H_9$ be the deletion graphs defined in
Corollary \ref{cor:nine-ratios}.
Among all graphs obtained from $K_n$ by deleting exactly $p$ edges,
the nine largest spanning-tree counts are attained uniquely up to
isomorphism by
\[
K_n-H_1,\ldots,K_n-H_9,
\]
in this order.
Namely,
\[
\tau(K_n-H_1)
>
\tau(K_n-H_2)
>
\cdots
>
\tau(K_n-H_9).
\]
Furthermore, if a $p$-edge deletion graph $H_F$
is not isomorphic to any of $H_1,\ldots,H_9$, then
\[
\tau(K_n-F)
<
\tau(K_n-H_9).
\]

\end{theorem}

\begin{proof}
Since $p\geq6$ and $n\geq2p$,
\[
n\geq12,
\qquad
p\leq\frac n2<n-1.
\]
The edge-connectivity of $K_n$ is $n-1$, so $K_n-F$ is connected
for every $p$-edge deletion set $F$.
Also, each $H_i$ has at most $2p$ vertices and is therefore realizable
as a subgraph of $K_n$.

Consider an arbitrary $p$-edge deletion graph $H_F$.
By Proposition \ref{prop:small-interaction},
if $a(F)\leq3$, then, up to isomorphism,
$H_F$ is one of $H_1,\ldots,H_9$.
Moreover, by Proposition \ref{prop:candidate-order},
\[
\mathcal{R}_n(H_1)
>
\mathcal{R}_n(H_2)
>
\cdots
>
\mathcal{R}_n(H_9).
\]

Since all deletion graphs under consideration have $p$ edges,
the normalization factor
\[
T_{n,p}
=
n^{\,n-p-2}(n-2)^p
\]
is common to all of them.
Therefore, from
\[
\tau(K_n-H_i)
=
T_{n,p}\mathcal{R}_n(H_i),
\]
we obtain
\[
\tau(K_n-H_1)
>
\tau(K_n-H_2)
>
\cdots
>
\tau(K_n-H_9).
\]

Finally, suppose that
$H_F$ is not isomorphic to any of $H_1,\ldots,H_9$.
By Proposition \ref{prop:small-interaction},
in this case we must have
\[
a(F)\geq4.
\]
Since $n\geq12$,
Proposition \ref{prop:large-interaction} gives
\[
\mathcal{R}_{n,p}(F)
<
\mathcal{R}_n(K_{1,3}).
\]
By Lemma \ref{lem:multiplicativity} and $\mathcal{R}_n(K_2)=1$,
\[
\mathcal{R}_n(K_{1,3})
=
\mathcal{R}_n(H_9).
\]
Therefore,
\[
\mathcal{R}_{n,p}(F)
<
\mathcal{R}_n(H_9).
\]
Multiplying again by the common normalization factor $T_{n,p}$ yields
\[
\tau(K_n-F)
<
\tau(K_n-H_9).
\]
Hence,
$H_1,\ldots,H_9$ uniquely attain, up to isomorphism,
the nine largest spanning tree counts in this order.

\end{proof}

\section{Logarithmic expansion and the local structure of the deletion graph}
\label{sec:log-expansion}
To clarify the role of the interaction number in the decrease of the
spanning-tree count, let
\[
A=Q-2I_p
\]
have eigenvalues $\mu_1,\ldots,\mu_p$, and set
\[
\mathcal{S}_{n,p}(F)
=
-\log\mathcal{R}_{n,p}(F).
\]
The logarithm admits the following expansion in terms of the traces of $A$.

\begin{proposition}
\label{prop:log-expansion}

Let $n\geq5$, and suppose that
$K_n-F$ is connected.
Then
\[
\mathcal{S}_{n,p}(F)
=
\sum_{k=2}^{\infty}
\frac{\operatorname{tr}(A^k)}
{k(n-2)^k}.
\]

\end{proposition}

\begin{proof}

Set $x=n-2$.
By Lemma \ref{lem:defect-moments},
\[
-2\leq\mu_\alpha<x.
\]
Since $n\geq5$, we have $x>2$, and therefore
\[
\left|
\frac{\mu_\alpha}{x}
\right|
<1
\qquad
(1\leq\alpha\leq p).
\]

By Lemma \ref{lem:normalized-tree},
\[
\mathcal{R}_{n,p}(F)
=
\prod_{\alpha=1}^{p}
\left(
1-\frac{\mu_\alpha}{x}
\right).
\]
Hence,
\[
\begin{aligned}
\mathcal{S}_{n,p}(F)
&=
-\sum_{\alpha=1}^{p}
\log\left(
1-\frac{\mu_\alpha}{x}
\right)
\\
&=
\sum_{\alpha=1}^{p}
\sum_{k=1}^{\infty}
\frac{\mu_\alpha^k}{kx^k}.
\end{aligned}
\]

Since the outer sum is taken over finitely many values of $\alpha$,
we may interchange the order of summation, obtaining
\[
\mathcal{S}_{n,p}(F)
=
\sum_{k=1}^{\infty}
\frac{1}{kx^k}
\sum_{\alpha=1}^{p}\mu_\alpha^k.
\]

Since $A$ is a real symmetric matrix,
\[
\sum_{\alpha=1}^{p}\mu_\alpha^k
=
\operatorname{tr}(A^k).
\]
Therefore,
\[
\mathcal{S}_{n,p}(F)
=
\sum_{k=1}^{\infty}
\frac{\operatorname{tr}(A^k)}{kx^k}.
\]

Furthermore, by Lemma \ref{lem:defect-moments},
\[
\operatorname{tr}A=0,
\]
so the first-order term vanishes.
Thus,
\[
\mathcal{S}_{n,p}(F)
=
\sum_{k=2}^{\infty}
\frac{\operatorname{tr}(A^k)}
{k(n-2)^k}.
\]

\end{proof}

For the quadratic term,
Lemma \ref{lem:interaction-trace} gives
\[
\operatorname{tr}A^2
=
2a(F).
\]

Therefore, the first term in
Proposition \ref{prop:log-expansion} is
\[
\frac{a(F)}{(n-2)^2}.
\]

We next determine the combinatorial meaning of the cubic term.
Let $t(H_F)$ denote the number of triangles contained in $H_F$.

\begin{lemma}
\label{lem:trace-cube}

\[
\operatorname{tr}A^3
=
6
\left\{
\sum_{v\in V(H_F)}
\binom{d_{H_F}(v)}{3}
-
t(H_F)
\right\}.
\]

\end{lemma}

\begin{proof}

Since all diagonal entries of $A$ are $0$,
\[
\operatorname{tr}A^3
=
\sum_{\alpha,\beta,\gamma=1}^{p}
A_{\alpha\beta}
A_{\beta\gamma}
A_{\gamma\alpha}
\]
can have a nonzero contribution only if
$\alpha,\beta,\gamma$ are distinct and the three deleted edges
\[
e_\alpha,\qquad
e_\beta,\qquad
e_\gamma
\]
pairwise share endpoints.

In a simple graph, there are only two configurations of three distinct edges
satisfying this condition.
Indeed, if three distinct edges are pairwise intersecting, then either
they are all incident with a common vertex, or they form a triangle.

First, suppose that the three edges are incident with a common vertex $v$.
If all three edges are oriented away from $v$,
the inner product of any two corresponding incidence vectors is $1$.
Therefore,
\[
A_{\alpha\beta}
A_{\beta\gamma}
A_{\gamma\alpha}
=
1.
\]

Changing the orientation of an edge changes the signs of the corresponding
row and column of $A$ simultaneously, and hence the value of this product
is independent of the choice of orientations.

The number of ways to choose three edges incident with $v$ is
\[
\binom{d_{H_F}(v)}{3},
\]
so the total number of unordered triples of this type is
\[
\sum_{v\in V(H_F)}
\binom{d_{H_F}(v)}{3}.
\]

Next, suppose that
$e_\alpha,e_\beta,e_\gamma$ form a triangle.
If the three edges are oriented cyclically around the triangle,
all three inner products are $-1$, and hence
\[
A_{\alpha\beta}
A_{\beta\gamma}
A_{\gamma\alpha}
=
-1.
\]

Again, this product is independent of the choice of edge orientations.
Thus, each triangle contributes $-1$.

Finally, each unordered triple
\[
\{e_\alpha,e_\beta,e_\gamma\}
\]
appears six times in the sum defining
$\operatorname{tr}A^3$,
corresponding to the six permutations of
$\alpha,\beta,\gamma$.
Therefore,
\[
\operatorname{tr}A^3
=
6
\left\{
\sum_{v\in V(H_F)}
\binom{d_{H_F}(v)}{3}
-
t(H_F)
\right\}.
\]

\end{proof}

Combining Proposition \ref{prop:log-expansion},
Lemma \ref{lem:interaction-trace}, and
Lemma \ref{lem:trace-cube},
we obtain the following local structural expansion.

\begin{corollary}
\label{cor:local-expansion}

Fix the isomorphism class of a deletion graph $H_F$, and regard it as
a subgraph of $K_n$ for all sufficiently large $n$.
Then, as $n\to\infty$,
\[
\begin{aligned}
-\log\mathcal{R}_{n,p}(F)
&=
\frac{a(F)}{(n-2)^2}
+
\frac{2}{(n-2)^3}
\left\{
\sum_{v\in V(H_F)}
\binom{d_{H_F}(v)}{3}
-
t(H_F)
\right\}
+
O\left((n-2)^{-4}\right).
\end{aligned}
\]

\end{corollary}

\begin{proof}

By Proposition \ref{prop:log-expansion},
\[
-\log\mathcal{R}_{n,p}(F)
=
\frac{\operatorname{tr}A^2}{2(n-2)^2}
+
\frac{\operatorname{tr}A^3}{3(n-2)^3}
+
O\left((n-2)^{-4}\right).
\]

Since $H_F$ is fixed,
$A$ and its eigenvalues are independent of $n$.
Let
\[
M=\max_{1\leq\alpha\leq p}|\mu_\alpha|.
\]
For all sufficiently large $n$, we have $M/(n-2)<1/2$, and hence
\[
\begin{aligned}
\left|
\sum_{k=4}^{\infty}
\frac{\operatorname{tr}(A^k)}{k(n-2)^k}
\right|
&\leq
p\sum_{k=4}^{\infty}
\left(
\frac{M}{n-2}
\right)^k
\\
&=
O\left((n-2)^{-4}\right).
\end{aligned}
\]

By Lemma \ref{lem:interaction-trace},
\[
\frac{\operatorname{tr}A^2}{2(n-2)^2}
=
\frac{a(F)}{(n-2)^2}.
\]

Also, by Lemma \ref{lem:trace-cube},
\[
\frac{\operatorname{tr}A^3}{3(n-2)^3}
=
\frac{2}{(n-2)^3}
\left\{
\sum_{v\in V(H_F)}
\binom{d_{H_F}(v)}{3}
-
t(H_F)
\right\}.
\]

This proves the assertion.

\end{proof}

Corollary \ref{cor:local-expansion} identifies $a(F)$ as the leading
local quantity governing the decrease from the matching case.
Among deletion graphs with the same interaction number, the next distinction
is given by
\[
\sum_{v\in V(H_F)}
\binom{d_{H_F}(v)}{3}
-
t(H_F).
\]
For example,
\[
a(K_3)=a(K_{1,3})=3.
\]

However, for $K_3$,
\[
\sum_{v\in V(K_3)}
\binom{d(v)}{3}=0,
\qquad
t(K_3)=1,
\]
and hence
\[
\sum_v\binom{d(v)}3-t(K_3)
=
-1.
\]

On the other hand, for $K_{1,3}$,
\[
\sum_{v\in V(K_{1,3})}
\binom{d(v)}3=1,
\qquad
t(K_{1,3})=0,
\]
and hence
\[
\sum_v\binom{d(v)}3-t(K_{1,3})
=
1.
\]

Therefore,
\[
\begin{aligned}
-\log\mathcal{R}_n(K_3)
&=
\frac{3}{(n-2)^2}
-
\frac{2}{(n-2)^3}
+
O\left((n-2)^{-4}\right),
\\
-\log\mathcal{R}_n(K_{1,3})
&=
\frac{3}{(n-2)^2}
+
\frac{2}{(n-2)^3}
+
O\left((n-2)^{-4}\right).
\end{aligned}
\]

This difference is consistent with the ordering
\[
\mathcal{R}_n(K_3)
>
\mathcal{R}_n(K_{1,3})
\]
obtained in Proposition \ref{prop:candidate-order}.

More generally, $A=Q-2I_p$ may be viewed as a signed adjacency matrix
of the line graph $L(H_F)$ of the deletion graph $H_F$.
Changing the orientation of a deleted edge corresponds to switching,
so $\operatorname{tr}A^k$ is independent of the choice of orientations
and may be interpreted as the signed sum of closed walks of length $k$
in $L(H_F)$.
Thus, the traces $\operatorname{tr}A^k$ with $k\geq4$ provide natural
quantities for distinguishing deletion graphs that cannot be distinguished
by the quadratic and cubic local structural quantities alone.
This perspective may also be useful in extending the higher-order ranking
beyond the ninth position.

\section{Conclusion and future directions}
\label{sec:conclusion}
We determined, up to isomorphism, the nine deletion graphs with the largest
spanning-tree counts among graphs obtained from $K_n$ by deleting $p$ edges,
under the assumptions $p\geq6$ and $n\geq2p$.
The main tool is a stability inequality for the normalized spanning-tree count
in terms of the interaction number $a(F)$.
Together with the classification of deletion graphs with $a(F)\leq3$
obtained in Part I, this estimate determines the first nine positions
without requiring an individual classification of the cases with $a(F)\geq4$.

The logarithmic expansion further shows that
\[
\operatorname{tr}A^2=2a(F)
\]
governs the leading deviation from the matching case, while the cubic term
involves
\[
\sum_{v\in V(H_F)}
\binom{d_{H_F}(v)}{3}
-
t(H_F).
\]
These higher-order traces may be useful for extending the ranking beyond
the ninth position.
It is also natural to determine how far the resulting ranking continues
to coincide with the Kirchhoff-index ranking obtained in Part I.

\section*{Statements and Declarations}
\noindent
\textbf{Funding.}
No funding was received to assist with the preparation of this manuscript.

\medskip
\noindent
\textbf{Competing interests.}
The authors have no relevant financial or non-financial interests to disclose.

\medskip
\noindent
\textbf{Data availability.}
No datasets were generated or analyzed during the current study.


\end{document}